\documentclass{article}
\usepackage{graphicx} 

\title{A Refinement of the Crank-Mex Theorem}
\makeatletter
\def\blfootnote{\xdef\@thefnmark{}\@footnotetext}
\makeatother
\author{by\\George E. Andrews and Moshe Newman}
\date{}

\usepackage{dsfont}
\usepackage{amsmath,amsthm,amsfonts}
\usepackage{tikz}
\allowdisplaybreaks

\usepackage[all]{xy}
\xyoption{arc}
\usepackage{enumerate, cite}

\numberwithin{equation}{section}

\newtheorem{theorem}{Theorem}
\newtheorem*{theorem*}{Theorem}

\theoremstyle{remark}

\theoremstyle{definition}

\newtheorem{lemma}[theorem]{Lemma}
\newtheorem*{lemma*}{Lemma}
\newtheorem*{conjecture*}{Conjecture}

\newcommand{\stirling}[2]{\genfrac{[}{]}{0pt}{}{#1}{#2}}

\DeclareMathOperator{\mex}{mex}

\begin{document}
\maketitle


\begin{abstract}
It is proved that the number of partitions of $n$ with odd mex and $k$ parts that aren't ones equals the number of partitions of $n$ with nonnegative crank and $k$ parts that aren't ones.\blfootnote{The first author is partially supported by Simons Foundation Grant 633284}
\end{abstract}

\section{Introduction}\label{Intro}

In 1944, Freeman Dyson \cite{5} conjectured the existence of a partition statistic, which he called the \underline{crank}. The crank was conjectured to provide a combinatorial explanation of the Ramanujan congruence
\[p(11n+6)\equiv0\pmod{11},\label{1.1}\]
where $p(n)$ is the number of partitions of $n$. In \cite{2} (c.f. \cite{6}), the crank was finally found. If we let $\omega(\pi)$ denote the number of ones appearing in the partition $\pi$ and $\eta(\pi)$ denote the number of parts of $\pi$ that are greater than $\omega(\pi)$, then the crank of $\pi,c(\pi)$, is defined by
\begin{equation}
\label{1.2} c(\pi)=\left\{
\begin{array}{cr}
    \lambda_1, & \text{if}~\omega(\pi)=0; \\
    \eta(\pi)-\omega(\pi), & \text{if}~\omega(\pi)>0,
\end{array}
\right.
\end{equation}
where $\lambda_1$ is the largest part of $\pi$.

In three independent papers \cite{4,7,11}, it was shown that the following is true.
\begin{theorem}\label{thm1}
    The number of partitions of $n$ with nonnegative crank is equal to the number of partitions of $n$ with odd $\mex$, where $\mex$ is the least positive integer that is not a part of $\pi$.
\end{theorem}
Our object in this paper is to prove the following:
\begin{theorem}\label{thm2}
    For $n\ge1$, the number of partitions of $n$ with nonnegative crank and $k$ parts that aren't ones equals the number of partitions of $n$ with odd $\mex$ and $k$ parts that aren't ones.
\end{theorem}
Note that Theorem 1 follows from Theorem \ref{thm2} by summing over all $k\ge0$ in Theorem \ref{thm2}.

In the next section, we identify the relevant generating functions and prove a necessary lemma.
\section{Background}\label{section2}
For brevity, we require the following notation:
\begin{align*}
    (a;q)_n&=\prod_{j=0}^{n-1}(1-aq^j)\quad (0\le n\le\infty).\\
    \stirling{A}{B}&=\left\{
\begin{array}{cr}
    0, & \text{if}~B<0~\text{or}~B>A; \vspace{5pt}\\
    \displaystyle\frac{(q;q)_A}{(q;q)_B(q;q)_{A-B}}, & \text{if}~0\le B\le A.
\end{array}
\right.
\end{align*}
Let $M(z,q)$ denote the generating function for partitions with odd $\mex$ where the exponent on $q$ is the number being partitioned and the exponent on $z$ is the number of parts that aren't ones. We present $M(z,q)$ with two terms. The first generates the partitions where the mex is 1, i.e. the partition has no ones at all. The entry $n$ in the second term generates the partitions with mex equal to $2n+1(n>0)$.
\begin{align*}\label{2.1}
    M(z,q)&=\frac{1}{(zq^2;q)_\infty}+\sum_{n\ge1}\frac{z^{2n-1}q^{1+2+\cdots+2n}(1-zq^{2n+1})}{(1-q)(zq^2;q)_\infty}\\
    &=\frac{1}{(zq^2;q)_\infty}\left(1+\sum_{n\ge1}\frac{z^{2n-1}q^{n(2n+1)}(1-zq^{2n+1})}{(1-q)}\right)\\
    &=\frac{1}{(zq^2;q)_\infty}\left(1+\sum_{n\ge2}\frac{(-1)^n z^{n-1} q^{\binom{n+1}{2}}}{(1-q)}\right).
\end{align*}
Let $K(z,q)$ denote the generating function for partitions with nonnegative crank where the exponent on $q$ is the number being partitioned and the exponent on $z$ is the number of parts that aren't ones.

To find $K(z,q)$ we must first look at the original generating function for the crank itself. Here the exponent on $y$ is the crank \cite{2}.

\begin{align*}
    \frac{(q;q)_\infty}{(yq;q)_\infty(q/y;q)_\infty}&=\frac{(1-q)(q^2;q)_\infty}{(yq;q)_\infty(q/y;q)_\infty}\\
    &=\frac{(1-q)}{(yq;q)_\infty}\left(1+\sum_{n=1}^\infty\frac{(yq;q)_n q^ny^{-n}}{(q;q)_n}\right)\\
    &\text{by \cite[p. 17, Th. 2.1, $t=qy^{-1},a=yq$]{1}}\\
    &=(1-q)\sum_{n\ge0}\frac{y^nq^n}{(q;q)_n}+(1-q)\sum_{n=1}^\infty\frac{(yq;q)_n q^ny^{-n}}{(q;q)_n(yq;q)_\infty}\\
    &=(1-q)+\sum_{n\ge1}\frac{y^nq^n}{(q^2;q)_{n-1}}+\sum_{n=1}^\infty\frac{q^ny^{-n}}{(q^2;q)_{n-1}(yq^{n+1};q)_\infty}\\
    &\text{(We note that the first sum gives the crank if the partition}\\
    &\text{has no ones and the $n^{th}$ term of the second sum gives the}\\
    &\text{partitions where there are $n$ ones)}\\
    &=(1-q)+\sum_{n\ge1}\frac{y^nq^n}{(q^2;q)_{n-1}}+\sum_{n=1}^\infty\frac{q^ny^{-n}}{(q^2;q)_{n-1}}\sum_{m=0}^\infty\frac{y^mq^{m(n+1)}}{(q;q)_m}.
\end{align*}
In this last expression, we must insert $z$ to keep track of parts $>1$ and we must only start the $m$-sum at $m=n$ in order to have nonnegative crank. Finally we set $y=1$ in that we are no longer interested in the crank but only that it is nonnegative.

Hence
\begin{align*}
    K(z,q)&=(1-zq)+\sum_{n\ge1}\frac{zq^n}{(zq^2;q)_{n-1}}+\sum_{n=1}^\infty\frac{q^n}{(zq^2;q)_{n-1}}\sum_{m=n}^\infty\frac{z^mq^{m(n+1)}}{(q;q)_m}\\
    &=(1-zq)+\sum_{n\ge1}\frac{zq^n}{(zq^2;q)_{n-1}}+\sum_{\underset{m\ge0}{n\ge1}}^\infty\frac{q^{n+(m+n)(n+1)}z^{m+n}}{(zq^2;q)_{n-1}(q;q)_{m+n}}.
\end{align*}
Note that we are only interested in partitions of $n$ for $n>1$. Thus the seemingly arbitrary $1-zq$ in $K(z,,q)$ does not affect cases for $n>1$, but is necessary for proving $M(z,q)=K(z,q)$.

We close this section with the following lemma:
\begin{lemma}
    \[\sum_{m=0}^H(-1)^mq^{\binom{m}{2}}\stirling{N}{m}=(-1)^Hq^{\binom{H+1}{2}}\stirling{N-1}{H}.\]
\end{lemma}
\begin{proof}
    Both sides equal 1 when $H=0$. Now assume the identity is true at $H-1$. Consequently,
    \begin{align*}
        \sum_{m=0}^H(-1)^mq^{\binom{m}{2}}\stirling{N}{m}&= \sum_{m=0}^{H-1}(-1)^mq^{\binom{m}{2}}\stirling{N}{m}+ (-1)^Hq^{\binom{H}{2}}\stirling{N}{H}\\
        &=(-1)^{H-1}q^{\binom{H}{2}}\stirling{N-1}{H-1}+(-1)^Hq^{\binom{H}{2}}\stirling{N}{H}\\
        &=(-1)^Hq^{\binom{H}{2}}\left(\stirling{N}{H}-\stirling{N-1}{H-1}\right)\\
        &=(-1)^Hq^{\binom{H}{2}+H}\stirling{N-1}{H}=(-1)^Hq^{\binom{H+1}{2}}\stirling{N-1}{H}
    \end{align*}
\end{proof}
\section{Proof of Theorem \ref{thm2}}\label{section3}
This only requires us to prove that
\[M(z,q)=K(z,q),\]
i.e.
\[\frac{1}{(zq^2;q)_\infty}\left(1+\sum_{n\ge2}\frac{
(-1)^nz^{n-1}q^{\binom{n+1}{2}}}{1-q}\right)\]
\[=1-zq+\sum_{n\ge1}\frac{zq^n}{(zq^2;q)_{n-1}}+\sum_{\underset{m\ge0}{n\ge1}}\frac{q^{n+(m+n)(n+1)}z^{m+n}}{(zq^2:q)_{n-1}(q;q)_{m+n}},\]
and multiplying both sides by $(zq^2;q)_\infty$ we see that our theorem is equivalent to
\begin{align*}
    &1+\sum_{n\ge2}\frac{
(-1)^nz^{n-1}q^{\binom{n+1}{2}}}{1-q}\\
&=(zq;q)_\infty+\sum_{n\ge1}zq^n(zq^{n+1};q)_\infty\\
&+\sum_{\underset{m\ge0}{n\ge1}}\frac{q^{n+(m+n)(n+1)}z^{m+n}(zq^{n+1};q)_\infty}{(q;q)_{m+n}}\\
&=\sum_{n\ge0}\frac{(-1)^nz^nq^{\binom{n+1}{2}}}{(q;q)_n}+\sum_{n\ge1}zq^n\sum_{s\ge0}\frac{(-1)^sz^sq^{s(n+1)+\binom{s}{2}}}{(q;q)_s}\\
&+\sum_{\underset{m\ge0}{n\ge1}}\frac{q^{n+(m+n)(n+1)}z^{m+n}}{(q;q)_{m+n}}\sum_{s\ge0}\frac{(-1)^sz^s q^{\binom{s+1}{2}+sn}}{(q;q)_s},
\end{align*}
where we have expanded the infinite products using \cite[p. 19, Cor. 2.2, eq. (2.2.6)]{1}.
To prove this last result we must show that the coefficients of $z^N$ on both sides are identical. It is clear that the constant term on each side is 1. So, we may take $N\ge1$.

The coefficient of $z^N$ on the left side is
\[\frac{(-1)^{N+1}q^{\binom{N+2}{2}}}{1-q}.\]
On the right we begin by noting that the coefficient of $z^N$ contributed by the first two expressions is
\[\frac{(-1)^Nq^{\binom{N+1}{2}}}{(q;q)_N}+\sum_{n\ge1}\frac{(-1)^{N-1}q^nq^{(N-1)(n+1)+\binom{N-1}{2}}}{(q;q)_{N-1}}.\]
Simplifying, we find
\begin{align*}
    &\frac{(-1)^Nq^{\binom{N+1}{2}}}{(q;q)_N}+\sum_{n\ge1}\frac{(-1)^{N-1}q^nq^{(N-1)(n+1)+\binom{N-1}{2}}}{(q;q)_{N-1}}\\
    &=\frac{(-1)^Nq^{\binom{N+1}{2}}}{(q;q)_N}+\frac{(-1)^{N-1}q^{\binom{N}{2}}\cdot q^N}{(q;q)_{N-1}(1-q^N)}\\
    &=\frac{(-1)^Nq^{\binom{N+1}{2}}}{(q;q)_N}-\frac{(-1)^Nq^{\binom{N+1}{2}}}{(q;q)_N}=0.
\end{align*}
Thus, since these two contributions cancel each other, we see that the entire contribution to the coefficient of $z^N$ on the right side comes from the third expression:
\begin{align*}
    &\sum_{\overset{\underset{m\ge0}{n\ge1}}{m+n\le N}}\frac{q^{\left(n+(m+n)(n+1)+\binom{N-n-m+1}{2}+(N-n-m)n\right)}(-1)^{N-n-m}}{(q;q)_{m+n}(q;q)_{N-m-n}}\\
    &=\sum_{n\ge1}\frac{q^{N+n+nN}}{(q;q)_N}\sum_{m=0}^{N-n}(-1)^{N-n-m}q^{\binom{N-n-m}{2}}\stirling{N}{n+m}\\
    &=\sum_{n\ge1}\frac{q^{N+n+nN}}{(q;q)_N}\sum_{m=0}^{N-n}(-1)^mq^{\binom{m}{2}}\stirling{N}{m}\\
    &=\sum_{n\ge1}\frac{q^{N+n+nN}}{(q;q)_N}(-1)^{N-n}q^{\binom{N-n+1}{2}}\stirling{N-1}{n-1}\quad(\text{by Lemma 3 with}~H=N-n)\\
    &=\frac{(-1)^Nq^{\binom{N+1}{2}+N}}{(q;q)_N}\sum_{n\ge1}(-1)^nq^{\binom{n+1}{2}}\stirling{N-1}{n-1}\\
    &=\frac{(-1)^{N-1}q^{\binom{N+1}{2}+N}}{(q;q)_N}\sum_{n\ge0}(-1)^nq^{\left(\binom{n+1}{2}+n+1\right)}\stirling{N-1}{n}\\
    &=\frac{(-1)^{N-1}q^{\binom{N+1}{2}+N}}{(q;q)_N}q(q^2;q)_{N-1}\quad(\text{by \cite[Ch. 3, p. 36, eq. (3.3.6)]{1}, the $q$-binomial theorem})\\
    &=\frac{(-1)^{N+1}q^{\binom{N+2}{2}}}{1-q}.
\end{align*}
This establishes the identity of the coefficients of $z^N$ on each side.

Thus Theorem \ref{thm2} is proved. $_\square$
\section{Conclusion}\label{conclusion}
Much has been written about the combinatorics of the crank and mex. The papers \cite{2,3,4,5,6,7,8,9,10,11} are all on this topic. We draw special attention to the deep work done in \cite{10}. We believe that it should be possible using Konan's method and results to provide a combinatorial proof of our theorem, but so far, this has eluded us.

\bibliographystyle{abbrv}

\noindent George E. Andrews\\
The Pennsylvania State University\\
University Park, PA 16802\\
gea1@psu.edu
\vspace{1em}\\
\noindent Moshe Newman\\
moshnoiman@gmail.com
\end{document}